\documentclass[11pt]{article}
\usepackage[utf8]{inputenc}
\usepackage[T1]{fontenc}
\usepackage{graphicx}
\usepackage{longtable}
\usepackage{wrapfig}
\usepackage{rotating}
\usepackage[normalem]{ulem}
\usepackage{amsmath}
\usepackage{amssymb}
\usepackage{capt-of}
\usepackage{hyperref}
\usepackage{amsmath,amsthm,amssymb,mathtools}
\usepackage{enumitem}
\usepackage{hyperref}
\usepackage{tikz}
\usepackage{tikz-cd}
\newcommand{\Meas}{\operatorname{Meas}}
\newtheorem{theorem}{Theorem}

\newtheorem{lemma}[theorem]{Lemma}

\author{Evan Misshula}
\date{\today}
\title{Signed Measures as the Linear Envelope of Positive Measures}
\hypersetup{
 pdfauthor={Evan Misshula},
 pdftitle={Signed Measures as the Linear Envelope of Positive Measures},
 pdfkeywords={},
 pdfsubject={},
 pdfcreator={Emacs 30.2.50 (Org mode 9.7.11)}, 
 pdflang={English}}
\begin{document}

\maketitle
\begin{abstract}
Signed measures are traditionally introduced as countably additive set
functions that admit both positive and negative values. The classical
Jordan decomposition theorem shows that every finite signed measure
may be expressed uniquely as the difference of two mutually singular
positive measures. While this theorem provides a structural
description of signed measures, it does not explain why signed
measures arise naturally from the theory of positive measures.

In this paper we give a categorical interpretation of this passage.
For every measurable space \(\Omega\), we show that the abelian group
of finite signed measures satisfies a universal property with respect
to the commutative monoid of finite positive measures: every additive
map from positive measures into an abelian group extends uniquely to a
group homomorphism on signed measures. In this sense, signed measures
provide the canonical additive extension of positive measure theory.

We compare this construction with classical Grothendieck completion,
clarifying both the similarities and the additional structure arising
from countable additivity and Jordan decomposition. The result places
the transition from positive to signed measures within a familiar
pattern of completion and linearization constructions and provides a
conceptual explanation for the central role of signed measures in
analysis and probability.
\end{abstract}
\section{Introduction}
\label{sec:org588845d}

Measure theory begins with positive measures, which assign
nonnegative values to measurable sets and admit a natural notion of
addition \cite{halmos1950,folland1999,cohn2013}. Positive
measures therefore form a commutative monoid under pointwise
addition. Many of the central constructions of measure theory,
however, require subtraction as well as addition. Signed measures,
charges, and related linear structures arise naturally throughout
analysis and probability theory, where they play a fundamental role in
decomposition theorems, integration theory, and functional-analytic
dualities \cite{halmos1950,folland1999}.

A fundamental result of measure theory, the Jordan decomposition
theorem, states that every finite signed measure may be written
uniquely as the difference of two mutually singular positive measures
\cite{halmos1950,folland1999,cohn2013}. Historically, this result
is closely related to Jordan's work on functions of bounded variation
\cite{jordan1881}. While this classical theorem provides a powerful
structural description, it invites a deeper question: what structural
necessity is fulfilled when passing from the positive cone to the
linear envelope?

A recurring pattern in mathematics is that positive structures
frequently need subtraction to become fully expressive. The passage
from the natural numbers to the integers is the classical
example. Similar phenomena occur throughout algebra, topology, and
K-theory, where universal completion constructions enlarge a positive
theory into a linear one \cite{maclane1998,awodey2010}. Viewed
categorically, the passage from positive to signed measures mirrors
this exact universal phenomenon. Since positive measures already carry
a natural commutative monoid structure under addition, this paper
demonstrates that signed measures arise as the object characterized by
a universal property to a corresponding universal completion
problem. Every additive map defined on positive measures extends
uniquely to signed measures; in this sense, they are the canonical
linearization of the additive structure already present in the
positive theory.
\subsection{Main theorem}
\label{sec:orgf41cb1c}

Let \(\Meas_{+}(\Omega)\) denote the commutative monoid of finite
positive measures on a measurable space \(\Omega\), and let
\(\Meas_{\pm}(\Omega)\) denote the abelian group of finite signed
measures.

\begin{theorem}[Main Theorem]
Let \(\Omega\) be a measurable space and let \(A\) be an abelian
group. Then every commutative-monoid homomorphism

\[
\Phi : \Meas_{+}(\Omega) \to A
\]

extends uniquely to a group homomorphism

\[
\widetilde{\Phi} : \Meas_{\pm}(\Omega) \to A .
\]

Consequently, \(\Meas_{\pm}(\Omega)\) is the universal additive
extension of \(\Meas_{+}(\Omega)\).
\label{thm:main}
\label{org3738817}
\end{theorem}

Informally, the theorem states that every construction compatible with
the finite additive structure of positive measures factors uniquely
through signed measures. Because we restrict our attention to the
monoid of finite positive measures, the completion behaves purely
algebraically without introducing topological convergence pathologies
in the arbitrary target group \(A\). Thus the passage

\[
\Meas_{+}(\Omega)
\longrightarrow
\Meas_{\pm}(\Omega)
\]

is canonical in the same sense that the passage from a commutative
monoid to its associated abelian group is canonical.

The classical Grothendieck construction associates an abelian group to
a commutative monoid by formally adjoining additive inverses.
Examples include the construction of the integers from the natural
numbers and the construction of \(K\)-groups from categories of vector
bundles.

The relationship between positive and signed measures exhibits many of
the same features. Positive measures form a commutative monoid,
signed measures form an abelian group, and Jordan decomposition
expresses signed measures as differences of positive measures.

Nevertheless, the measure-theoretic setting contains additional
structure arising from sigma-additivity and mutual singularity. For
this reason, signed measures should not be identified naively with an
abstract Grothendieck completion. Rather, they realize a
measure-theoretic instance of the same universal phenomenon. One of
the aims of this paper is to make this analogy precise while
clarifying both its strengths and its limitations.

The theory of positive measures, signed measures, and Jordan
decomposition is classical
\cite{halmos1950,folland1999,cohn2013}. The contribution of the
present paper is not a new decomposition theorem but rather a new
interpretation of the existing theory.

Specifically, we show that finite signed measures satisfy a universal
property analogous to additive completion. While Jordan decomposition
is classical, the contribution of this paper is the identification and
proof of the associated universal property. This characterization
places the passage from positive measures to signed measures within
the broader context of completion and linearization constructions in
category theory.

Section 2 develops the additive structure of positive measures and
establishes the categorical framework in which the construction takes
place. Section 3 reviews signed measures and Jordan decomposition.
Section 4 proves the main universal property. Section 5 compares the
result with classical Grothendieck completion and explains the sense
in which signed measures provide a linearization of positive measure
theory. The remaining sections discuss consequences, connections with
categorical probability, and directions for future work.
\section{Positive Measures as an Additive Theory}
\label{sec:org2f00c6b}

Let \((\Omega,\mathcal F)\) be a measurable space; that is,
\(\mathcal F\) is a \(\sigma\)-algebra of subsets of \(\Omega\).
A finite positive measure on \((\Omega,\mathcal F)\) is a function

\[
\mu : \mathcal F \to [0,\infty)
\]

such that

\[
\mu(\varnothing)=0
\]

and

\[
\mu\left(\bigcup_{n=1}^{\infty} E_n\right)=\sum_{n=1}^{\infty}\mu(E_n)
\]

whenever \((E_n)_{n\ge 1}\) is a countable family of pairwise disjoint
measurable sets \cite{halmos1950,folland1999,cohn2013}.  A
measure is finite if \(\mu(\Omega)<\infty\).

Positive measures are the basic objects of classical measure theory.
They assign nonnegative size, mass, probability, or volume to
measurable sets and form the starting point for the development of
integration and probability.

Throughout this paper we write

\[
\Meas_{+}(\Omega)
\]

for the commutative monoid of finite positive measures on
\(\Omega\).

Positive measures admit a natural addition. Given

\[
\mu,\nu \in \Meas_{+}(\Omega),
\]

their sum is defined pointwise by

\[
(\mu+\nu)(E)
=
\mu(E)+\nu(E).
\]

for every measurable set (E).

Because sums of positive measures are again positive measures, this
operation is closed on \(\Meas_{+}(\Omega)\). The zero measure serves
as an identity element, and associativity and commutativity follow
immediately from the corresponding properties of addition in
\([0,\infty]\).

Consequently,

\[
(\Meas_{+}(\Omega),+,0)
\]

forms a commutative monoid.

This observation is elementary, but it is the structural starting
point for everything that follows. The collection of positive
measures is not merely a set of objects; it already carries a
nontrivial additive structure.
\subsection{Pushforwards and functoriality}
\label{sec:org3f2753a}

A measurable map

\[
f : \Omega \to \Omega'
\]

transports measures by pushforward. Given a positive measure
\(\mu \in \Meas_{+}(\Omega)\), its pushforward (sometimes denoted
\(f_{\sharp}\mu\) in the measure-theoretic literature)is defined by

\[
f_{*}\mu(E)=\mu(f^{-1}(E))
\]

for every measurable subset \(E\subseteq \Omega'\) \cite{folland1999,villani2009}.  

Pushforward preserves positivity and countable additivity, so
\(f_{*}\mu\) is again a positive measure
\cite{halmos1950,folland1999}.

Moreover,

\[
(id_{\Omega})_{*}=id
\]

and

\[
(g\circ f)_{*}=g_{*}\circ f_{*}.
\]

Thus the assignments

\[
\Omega \longmapsto \Meas_{+}(\Omega),
\qquad
f \longmapsto f_{*},
\]

define a functor from the category of measurable spaces to the
category of commutative monoids.  Although each \(\Meas_{+}(\Omega)\)
carries additional structure arising from countable additivity, only
the underlying commutative-monoid structure will be needed for the
universal property established below.

The preceding discussion leaves only one essential deficiency:
positive measures can be added but not subtracted. Any additive theory
built from positive measures therefore points beyond the cone
\(\Meas_{+}(\Omega)\) toward a larger setting in which additive
inverses exist.

The next section shows that the Jordan decomposition theorem provides
exactly this enlargement. Signed measures arise as the minimal
extension in which subtraction becomes available while preserving the
additive structure of positive measures.
\section{Constructing Signed Measures}
\label{sec:orgfa5130e}

\subsection{Jordan decomposition}
\label{sec:org09e60b9}

The absence of additive inverses suggests enlarging the collection of
positive measures to a larger class in which subtraction becomes
possible. The resulting objects are the signed measures of classical
measure theory.

A finite signed measure on \((\Omega,\mathcal F)\) is a countably
additive function

\[
\mu:\mathcal F\to\mathbb R
\]

whose total variation is finite and that may take both positive and
negative values \cite{halmos1950,folland1999,cohn2013}.

The fundamental structural result governing signed measures is the
Jordan decomposition theorem. In its modern measure-theoretic form,
the theorem states that every finite signed measure may be expressed
uniquely as the difference of two mutually singular positive measures
\cite{halmos1950,folland1999,cohn2013}. Historically, this
result is closely related to Jordan's work on functions of bounded
variation \cite{jordan1881}.

\begin{theorem}[Jordan Decomposition]
Every finite signed measure (\(\mu\)) admits a unique decomposition

\[
\mu=\mu^{+}-\mu^{-},
\]

where \(\mu^{+}\) and \(\mu^{-}\) are finite positive measures that
are mutually singular.
\end{theorem}

This theorem expresses every signed measure in terms of positive
measure data. Every signed measure may be reconstructed from a
positive part and a negative part, and the measures
\(\mu^{+}\) and \(\mu^{-}\) are uniquely determined by
\(\mu\).

This perspective recontextualizes signed measures as formal
differences of positive components, directly mirroring the
construction of the integers from the natural numbers or the formation
of abstract Grothendieck groups
\cite{maclane1998,awodey2010,leinster2014}.

Indeed, every signed measure admits a representation

\[
\mu=\mu^{+}-\mu^{-},
\]

while every positive measure may be regarded as a special case by
taking \(\mu^{-}=0\).

This observation closely resembles familiar constructions elsewhere in
mathematics. Integers may be represented as differences of natural
numbers, and elements of Grothendieck groups may be represented as
differences of elements of an underlying commutative monoid
\cite{maclane1998,awodey2010,leinster2014}.

The analogy should not be taken too literally. Signed measures belong
to a richer measure-theoretic setting than an abstract Grothendieck
completion, incorporating countable additivity and the structure
encoded by the Jordan decomposition.

Nevertheless, the representation

\[
\mu=\mu^{+}-\mu^{-}
\]

which strongly suggests that signed measures are serving the role of
additive inverses for positive measures. The next section shows that
this intuition is not merely heuristic but is captured by a precise
universal property.

The pushforward construction extends immediately from positive
measures to signed measures. Given a measurable map

\[
f:\Omega\to\Omega',
\]

the pushforward of a signed measure \(\mu\) is defined by

\[
f_{*}\mu(E) = \mu(f^{-1}(E))
\]

for every measurable set \(E\in\mathcal F'\).

Since pushforward is additive,

\[
f_{*}(\mu+\nu) = f_{*}\mu+f_{*}\nu,
\]

it preserves the additive structure of signed measures. Moreover,

\[
(id_{\Omega})_{*}=id
\]

and

\[
(g\circ f)_{*} = g_{*}\circ f_{*}.
\]

Thus the constructions developed for positive measures extend
unchanged to signed measures.

The compatibility of pushforward with Jordan decomposition will play
an important role below. If

\[
\mu=\mu^{+}-\mu^{-},
\]

then

\[
f_{*}\mu = f_{*}\mu^{+} - f_{*}\mu^{-}.
\]

In particular, pushforward preserves the representation of a signed
measure as a difference of positive measures.

The collection \(\Meas_{\pm}(\Omega)\) carries a natural additive
structure. Given signed measures \(\mu\) and \(\nu\), define

\[
(\mu+\nu)(E) = \mu(E)+\nu(E).
\]

The zero measure serves as an identity element, and every signed
measure possesses an additive inverse given by

\[
(-\mu)(E) = -\mu(E).
\]

Consequently,

\[
(\Meas_{\pm}(\Omega),+,0)
\]

forms an abelian group.

Together with pushforward, these groups assemble into a functor

\[
\Meas_{\pm} : \mathbf{Meas}  \longrightarrow \mathbf{Ab}.
\]

The contrast with the previous section is now apparent. Positive
measures form commutative monoids, whereas signed measures form
abelian groups. The Jordan decomposition theorem explains how every
element of the latter may be expressed in terms of the former.

The next section shows that this relationship is not merely a useful
representation theorem. It is characterized by a universal property
that identifies signed measures as the canonical additive extension of
positive measures.
\section{The Universal Property}
\label{sec:org54a03ff}

The preceding sections have shown that positive measures form
commutative monoids, while signed measures form abelian groups.
Moreover, every signed measure admits a unique Jordan decomposition

\[
\mu=\mu^{+}-\mu^{-}.
\]

The Jordan decomposition theorem provides a distinguished
representation of every signed measure as a difference of positive
measures. For the universal property proved in the next section,
however, it will be useful to understand a more general phenomenon.

A signed measure may admit many representations as a difference of
positive measures. The following observation shows that any additive
construction defined on positive measures assigns the same value to all
such representations.

\begin{lemma}[Independence of Positive Representation]
\label{lem:independence}
Let \(\mu\in\Meas_{\pm}(\Omega)\), and suppose that

\[
\mu=\alpha-\beta=\alpha'-\beta'
\]

are two representations of \(\mu\) as a difference of finite positive
measures. Then

\[
\alpha+\beta'=\alpha'+\beta.
\]

Consequently, if

\[
\Phi:\Meas_{+}(\Omega)\to A
\]

is a commutative-monoid homomorphism into an abelian group \(A\), then

\[
\Phi(\alpha)-\Phi(\beta) = \Phi(\alpha')-\Phi(\beta').
\]
\label{org06de750}
\end{lemma}

\begin{proof}
Since

\[
\alpha-\beta=\alpha'-\beta',
\]

we have equality of signed measures

\[
\alpha+\beta'=\alpha'+\beta.
\]

Both \(\alpha+\beta'\) and \(\alpha'+\beta\) are finite positive
measures. Since

\[
\alpha+\beta'=\alpha'+\beta,
\]

applying the monoid homomorphism \(\Phi\) gives

\[
\Phi(\alpha+\beta') = \Phi(\alpha'+\beta).
\]

By additivity of \(\Phi\), this becomes

\[
\Phi(\alpha)+\Phi(\beta') = \Phi(\alpha')+\Phi(\beta).
\]

Since \(A\) is an abelian group, subtracting
\(\Phi(\beta)+\Phi(\beta')\) from both sides yields

\[
\Phi(\alpha)-\Phi(\beta) = \Phi(\alpha')-\Phi(\beta').
\]
\end{proof}

Because Lemma\textasciitilde{}\ref{lem:independence} shows that the value of
\(\widetilde{\Phi}\) depends only on the underlying signed measure
and not on a particular representation as a difference of positive
measures, the failure of pushforward to preserve mutual singularity
is irrelevant. Consequently, the extension remains compatible with
pushforward even when \(f_{*}\mu^{+}\) and \(f_{*}\mu^{-}\) are not
the Jordan components of \(f_{*}\mu\).
\subsection{Statement of the universal property}
\label{sec:orgeccd557}

The relationship between positive and signed measures is summarized
by the following factorization property.

\[
\begin{tikzcd}[row sep=large, column sep=large]
\Meas_{+}(\Omega)
  \arrow[r, "\iota"]
  \arrow[rd, "\Phi"']
&
\Meas_{\pm}(\Omega)
  \arrow[d, dashed, "\widetilde{\Phi}"]
\\
&
A
\end{tikzcd}
\]

Given an abelian group \(A\) and a commutative-monoid homomorphism

\[
\Phi:\Meas_{+}(\Omega)\to A,
\]

the question is whether \(\Phi\) factors uniquely through the
inclusion of positive measures into signed measures. The following
theorem shows that it does.

\begin{theorem}[Universal Property]
\label{thm:universal}
Let \(\Omega\) be a measurable space and let \(A\) be an abelian
group. For every commutative-monoid homomorphism

\[
\Phi:\Meas_{+}(\Omega)\to A,
\]

there exists a unique group homomorphism

\[
\widetilde{\Phi}:\Meas_{\pm}(\Omega)\to A
\]

such that

\[
\widetilde{\Phi}(\mu)=\Phi(\mu)
\]

for every positive measure
\(\mu\in\Meas_{+}(\Omega)\).
\label{orgf73a173}
\end{theorem}

Thus every additive construction defined on positive measures factors
uniquely through signed measures.
\subsection{The Extension Forced by Jordan Decomposition}
\label{sec:org46bee0a}

Let

\[
\Phi:\Meas_{+}(\Omega)\to A
\]

be a commutative-monoid homomorphism.

The Jordan decomposition theorem states that every signed measure
admits a unique representation

\[
\mu=\mu^{+}-\mu^{-}.
\]

Since \(A\) is an abelian group, subtraction is available in the
codomain. Consequently, any additive extension of \(\Phi\) must
satisfy

\[
\widetilde{\Phi}(\mu) = \Phi(\mu^{+}) - \Phi(\mu^{-}).
\]

Jordan decomposition therefore does not merely suggest an extension;
it determines the only possible additive extension.
\subsection{Proof of the Universal Property}
\label{sec:orgf995b83}

\begin{proof}
Define

\[
\widetilde{\Phi}(\mu) = \Phi(\mu^{+}) - \Phi(\mu^{-})
\]

for every signed measure

\[
\mu=\mu^{+}-\mu^{-}.
\]

Because the Jordan decomposition is unique, this definition is
well-defined.

If \(\mu\) is positive, then

\[
\mu=\mu-0,
\]

and therefore

\[
\widetilde{\Phi}(\mu) = \Phi(\mu)-\Phi(0) = \Phi(\mu).
\]

Thus \(\widetilde{\Phi}\) extends \(\Phi\).

To verify additivity, write

\[
\mu=\mu^{+}-\mu^{-},
\qquad
\nu=\nu^{+}-\nu^{-}.
\]

Then

\[
\mu+\nu = (\mu^{+}+\nu^{+}) - (\mu^{-}+\nu^{-}),
\]

which expresses \(\mu+\nu\) as a difference of positive measures.
Although this need not be the Jordan decomposition of \(\mu+\nu\), the
preceding lemma shows that the value of \(\widetilde{\Phi}\) is
independent of the chosen positive representation. We may therefore
compute using this representation.

Using the additivity of \(\Phi\) on positive measures, we obtain

\begin{equation*}
\begin{aligned}
\widetilde{\Phi}(\mu+\nu)
&=
\Phi(\mu^{+}+\nu^{+})
-
\Phi(\mu^{-}+\nu^{-}) \\
&=
\Phi(\mu^{+})
+
\Phi(\nu^{+})
-
\Phi(\mu^{-})
-
\Phi(\nu^{-}) \\
&=
\bigl(\Phi(\mu^{+})-\Phi(\mu^{-})\bigr)
+
\bigl(\Phi(\nu^{+})-\Phi(\nu^{-})\bigr) \\
&=
\widetilde{\Phi}(\mu)
+
\widetilde{\Phi}(\nu).
\end{aligned}
\end{equation*}

Hence \(\widetilde{\Phi}\) is a group homomorphism.

Finally, suppose

\[
\Psi:\Meas_{\pm}(\Omega)\to A
\]

is another group homomorphism extending \(\Phi\). For every signed
measure

\[
\mu=\mu^{+}-\mu^{-},
\]

we have

\begin{equation*}
\begin{aligned}
\Psi(\mu)
&=
\Psi(\mu^{+}-\mu^{-}) \\
&=
\Psi(\mu^{+})
-
\Psi(\mu^{-}) \\
&=
\Phi(\mu^{+})
-
\Phi(\mu^{-}) \\
&=
\widetilde{\Phi}(\mu).
\end{aligned}
\end{equation*}

Therefore \(\Psi=\widetilde{\Phi}\), proving uniqueness.
\end{proof}

The theorem shows that signed measures are characterized by the
requirement that additive maps defined on positive measures extend
uniquely to an abelian group. In this sense, the passage from
positive measures to signed measures is determined by a universal
property rather than by a particular construction.
\subsection{Compatibility with pushforward}
\label{sec:orgcd789fc}

The universal property is compatible with measurable maps.

Let

\[
f:\Omega\to\Omega'
\]

be measurable. Pushforward induces commutative-monoid homomorphisms

\[
f_*:\Meas_{+}(\Omega)\to\Meas_{+}(\Omega')
\]

and group homomorphisms

\[
f_*:\Meas_{\pm}(\Omega)\to\Meas_{\pm}(\Omega').
\]

If

\[
\mu=\mu^{+}-\mu^{-},
\]

then

\[
f_*\mu = f_*\mu^{+} - f_*\mu^{-}.
\]

Although \(f_*\mu^{+}\) and \(f_*\mu^{-}\) need not remain mutually
singular, the extension construction depends only on the resulting
signed measure and not on a particular positive representation.
Consequently, the correspondence

\[
\Phi \longmapsto \widetilde{\Phi}
\]

is compatible with pushforward and therefore natural in
\(\Omega\).

Thus the universal property is compatible with the functorial
structure developed in the previous sections. Indeed, both constructions assign to a representation

\[
\mu=\alpha-\beta
\]

the value

\[
\Phi(f_*\alpha)-\Phi(f_*\beta),
\]

and the preceding lemma guarantees independence from the chosen
representation.
\subsection{Example: Integration against a bounded function}
\label{sec:org4dac7c8}

Let \(g:\Omega\to\mathbb R\) be a bounded measurable function. For
every positive measure \(\mu\), define

\[
\Phi_g(\mu) = \int_{\Omega} g\,d\mu.
\]

Since integration is additive in the measure argument,

\[
\Phi_g(\mu+\nu) = \Phi_g(\mu)+\Phi_g(\nu),
\]

so \(\Phi_g\) is a commutative-monoid homomorphism

\[
\Phi_g:\Meas_{+}(\Omega)\to\mathbb R.
\]

Since \(\Phi_g\) is a commutative-monoid homomorphism,
Theorem\textasciitilde{}\ref{thm:universal} implies the existence of a unique
group homomorphism

\[
\widetilde{\Phi}_{g}:\Meas{\pm}(\Omega)\to\mathbb R
\]

extending \(\Phi_g\).

If

\[
\mu=\mu^{+}-\mu^{-},
\]

then the extension is given by

\[
\widetilde{\Phi}_g(\mu) =  \int_{\Omega} g\,d\mu^{+} - \int_{\Omega} g\,d\mu^{-}.
\]

This is precisely the classical definition of integration with
respect to a signed measure.

Thus the familiar extension of integration from positive measures to
signed measures is not merely a convenient definition. It is the
unique extension compatible with the additive structure of positive
measures, and therefore an instance of the universal property
established above.
\section{Relation to Grothendieck Completion}
\label{sec:orga013736}

The universal property established in the previous section places the
passage from positive measures to signed measures within a broader
family of completion constructions that occur throughout algebra and
category theory. In particular, it invites comparison with the
classical Grothendieck construction, which associates an abelian group
to a commutative monoid by adjoining additive inverses
\cite{maclane1998,awodey2010,leinster2014}.

A signed measure is not merely an arbitrary extension of the notion of
measure. Rather, it provides a systematic way to interpret formal
differences of positive measures as mathematical objects in their own
right.
\subsection{Comparison with the classical construction}
\label{sec:org161bcbc}

Let \(M\) be a commutative monoid. The Grothendieck construction
produces an abelian group \(G(M)\) together with a monoid homomorphism

\[
\iota : M \to G(M)
\]

such that every monoid homomorphism from \(M\) into an abelian group
factors uniquely through \(G(M)\)
\cite{maclane1998,awodey2010}.

The universal property established in
Theorem \ref{thm:universal} exhibits a closely analogous phenomenon.
For a measurable space \(\Omega\), the commutative monoid

\[
\Meas_{+}(\Omega)
\]

of finite positive measures admits a canonical embedding into the
abelian group

\[
\Meas_{\pm}(\Omega),
\]

and every additive map from positive measures into an abelian group
extends uniquely to signed measures.

The analogy becomes particularly transparent through Jordan
decomposition. Every signed measure may be represented as

\[
\mu=\mu^{+}-\mu^{-},
\]

just as an element of a Grothendieck group may be represented as a
difference of elements of the underlying monoid. From this
perspective, signed measures play the role of additive inverses for
positive measures, while the universal property identifies
\(\Meas_{\pm}(\Omega)\) as the canonical additive extension of
\(\Meas_{+}(\Omega)\).
\subsection{Limits of the analogy}
\label{sec:org801c579}

The analogy with Grothendieck completion is strong but should not be
interpreted as an exact identification.

The classical Grothendieck construction is purely algebraic. Starting
from a commutative monoid, it freely adjoins additive inverses and
produces an abelian group characterized by a universal property. The
construction depends only on the underlying additive structure of the
monoid.

The passage from positive measures to signed measures exhibits the
same universal behavior, but it occurs within a substantially richer
environment. Positive measures are not merely elements of a
commutative monoid. They are countably additive set functions defined
on measurable spaces, and their additive structure is constrained by
the measure-theoretic requirement of countable additivity.

Moreover, the Jordan decomposition theorem provides more than the
existence of additive inverses. Every signed measure admits a
distinguished decomposition

\[
\mu=\mu^{+}-\mu^{-},
\]

in which the positive and negative parts are mutually singular.
Ordinary Grothendieck completion provides no analogous canonical
representation. In general, an element of a Grothendieck group may be
represented in many different ways as a formal difference, with no
preferred choice.

Thus signed measures possess additional structure beyond that required
by the universal property alone. The universal property identifies
their additive role, while countable additivity and Jordan
decomposition supply the analytical content that makes the theory
useful.

For this reason, signed measures should not be viewed as merely an
instance of Grothendieck completion. Rather, they provide a
measure-theoretic realization of the same universal phenomenon,
enriched by structures that have no counterpart in the purely
algebraic setting.
\section{Discussion and Further Directions}
\label{sec:orgd5c45cb}

The integration example illustrates a broader phenomenon. Classical
constructions that extend naturally from positive measures to signed
measures often do so because signed measures satisfy the universal
property proved in Theorem \ref{thm:universal}. Every additive map
defined on positive measures extends uniquely to signed measures, so
the passage

\[
\Meas_{+}(\Omega) \longrightarrow \Meas_{\pm}(\Omega)
\]

is determined by an abstract categorical requirement rather than by a
particular analytical construction.

From this perspective, the Jordan decomposition theorem acquires a new
interpretation. Classically, Jordan decomposition is viewed as a
structural description of signed measures. The universal property
shows that it also provides the mechanism through which additive
inverses enter measure theory. In this sense, signed measures may be
understood as the canonical additive completion of positive measure
theory.
\subsection{Relation to categorical probability}
\label{sec:org8653ac2}

Measures and probability measures have played a central role in the
development of categorical probability. Examples include the Giry
monad \cite{giry1982}, categorical treatments of distributions
\cite{kock2006}, and more recent approaches based on Markov
categories and synthetic probability theory
\cite{fritz2019,perrone2022}.

Most of these frameworks begin with positive measures, probability
measures, or stochastic maps as primitive objects. The present result
suggests that the passage to signed measures is not merely an
additional construction but reflects a universal completion process
already implicit in the additive structure of positive measure
theory.

Accordingly, the universal property established here may be viewed as
complementary to existing categorical approaches to probability. It
identifies a structural principle underlying the transition from
positive to signed measures and thereby provides a bridge between
classical measure theory and categorical formulations of probability.
\subsection{Signed kernels and categorical probability}
\label{sec:orge0bf234}

The theorem proved in this paper concerns measures on a fixed
measurable space. A natural next step is to consider kernels.

Positive kernels, Markov kernels, and stochastic maps play a central
role throughout categorical probability. Since signed measures arise
as a canonical additive extension of positive measures, it is natural
to ask whether signed kernels arise as a corresponding additive
completion of positive kernels.

Such a result would provide a linearized setting for categorical
probability analogous to the role played by signed measures in
classical analysis. It may also offer a natural route toward enriched
or additive versions of Markov categories in which subtraction is
available in addition to convex combination.

Whether a universal characterization of signed kernels exists remains
an open question.
\subsection{Open questions}
\label{sec:org80c3ec5}

Several natural directions remain open. First, it would be desirable
to determine the precise relationship between the universal property
established here and the classical Grothendieck construction. While
the analogy is strong, the measure-theoretic setting incorporates
additional structure through countable additivity and Jordan
decomposition.

Second, one may ask whether other foundational constructions of
measure theory admit similar universal characterizations. A
particularly interesting candidate is the Carathéodory extension
theorem, whose role in constructing measures from premeasures
suggests a possible categorical interpretation.

Another natural question is whether the finiteness assumption may be
removed. The finite setting is particularly well suited to the
present treatment because every signed measure admits a Jordan
decomposition into finite positive measures, and the resulting
universal property may be formulated entirely in algebraic terms.

For more general classes of measures, additional analytical issues
arise. One must specify appropriate notions of convergence and
compatibility with infinite values, and the relevant categories may
require topological or measure-theoretic structure beyond that of
commutative monoids and abelian groups. Determining the correct
categorical framework for such extensions remains an interesting
direction for future investigation.

Finally, the emergence of categorical probability over the past
several decades raises the broader question of which aspects of
measure theory are most naturally understood through universal
properties. The result proved here suggests that at least some of the
central constructions of classical measure theory may admit such an
interpretation.

\clearpage
\bibliographystyle{alpha}
\bibliography{refs/bib/signed_measures}
\end{document}